\documentclass[a4paper, 11pt]{amsart}
\usepackage[a4paper, hmargin={2.7cm,2.7cm},vmargin={2.7cm,2.7cm}]{geometry}
 \newtheorem{thm}{Theorem}[section]
 \newtheorem{cor}[thm]{Corollary}
 
 \newtheorem{prop}[thm]{Proposition}
 \theoremstyle{definition}
 
 \theoremstyle{remark}
 \newtheorem{rem}[thm]{Remark}
 \newtheorem{ex}[thm]{Example}
 \numberwithin{equation}{section}

\DeclareMathOperator*{\vnl}{VN_L(G)}
\DeclareMathOperator*{\vnr}{VN_R(G)}
\DeclareMathOperator*{\vnrp}{VN^+_R(G)}
\renewcommand{\asymp}{\sim}
\newcommand{\hdim}{Q_*}

\DeclareMathOperator{\supp}{supp}

\usepackage[english]{babel}
\usepackage{amsmath}
\usepackage{amssymb}
\usepackage{amsfonts}
\usepackage{amsthm}
\usepackage{enumerate}
\usepackage{url}
\usepackage{mathdots}
\usepackage{etoolbox}
\usepackage{xcolor}
\usepackage{enumerate}

\usepackage[bookmarksnumbered, colorlinks = true, linkcolor = blue, citecolor = blue, plainpages]{hyperref}

\begin{document}

\title[$L^p$-$L^q$ norms of spectral multipliers]{An update on the $L^p$-$L^q$ norms of spectral \\ multipliers on unimodular Lie groups}

\author{David Rottensteiner}

\address{
Department of Mathematics: Analysis,
Logic and Discrete Mathematics \\
Ghent University \\
Krijgslaan 281, S8 \\
9000 Gent \\
Belgium}

\email{david.rottensteiner@ugent.be}


\author{Michael Ruzhansky}

\address{
Department of Mathematics: Analysis,
Logic and Discrete Mathematics \\
Ghent University \\
Krijgslaan 281, S8 \\
9000 Gent \\
Belgium \\
and \\
School of Mathematical Sciences \\
Queen Mary University of London \\
Mile End Road \\
London E1 4NS \\
United Kingdom}

\email{michael.ruzhansky@ugent.be}



\date{}

\maketitle

\begin{abstract}
This note gives a wide-ranging update on the multiplier theorems by Akylzhanov and the second author [J. Funct. Anal., 278 (2020), 108324]. The proofs of the latter crucially rely on $L^p$-$L^q$ norm estimates for spectral projectors
of left-invariant weighted subcoercive operators on unimodular Lie groups, such as Laplacians, sub-Laplacians and Rockland operators.
By relating spectral projectors to heat kernels, explicit estimates of the $L^p$-$L^q$ norms can be immediately exploited for a much wider range of (connected unimodular) Lie groups and operators than previously known. The comparison with previously established bounds by the authors show that the heat kernel estimates are sharp.
As an application, it is shown that several consequences of the multiplier theorems, such as time asymptotics for the $L^p$-$L^q$ norms of the heat kernels and Sobolev-type embeddings, 
are then automatic for the considered operators.
\end{abstract}

Let $G$ be a connected unimodular Lie group and let  $\lambda_G$ and $\rho_G$ be the left and right
regular representations, respectively, on $L^2 (G)$. Denote the associated 
left and right group von Neumann algebras by
\begin{align*}
\vnl := \lambda_G (G)'' \quad \text{and} \quad \vnr := \rho_G (G)'',
\end{align*}
respectively. Given any left-invariant operator $T \in \mathcal{L}(L^2 (G))$, i.e.,
\begin{align*}
T \lambda_G (x) = \lambda_G (x) T, \quad x \in G,
\end{align*}
the double commutant theorem implies that $T \in \vnr$. 
By the Schwartz kernel theorem, there exists a unique distribution $K_T \in \mathcal{D}' (G)$
such that 
\begin{align*}
T f(x) = f \ast K_T (x) = \int_G K_T (y^{-1} x) f(y) \; dy, \quad x \in G,
\end{align*}
for $f \in \mathcal{D}(G)$.
A trace $\tau$ on the positive part $\vnrp$ of $\vnr$ can, for example, be defined by
\begin{align}
\begin{split} \label{eq:trace}
\tau(T) := 
\begin{cases}
\| K_{T^{1/2}} \|_{L^2(G)}^2 \quad & \text{if} \; K_{T^{1/2}} \in L^2 (G), \\
\infty & \text{otherwise} .
\end{cases}
\end{split}
\end{align}

Let $\mathcal{L}$ be a (possibly unbounded) positive self-adjoint operator 
on $L^2 (G)$. Then by the spectral theorem, 
the operator $\mathcal{L}$ admits a spectral resolution
$\mathcal{L} = \int_0^\infty \lambda dE_{\lambda}$.
Given a Borel measurable function $\varphi : \mathbb{R} \to \mathbb{C}$, the operator
\begin{align*}
\varphi (\mathcal{L}) = \int_0^\infty \varphi(\lambda) dE_{\lambda}
\end{align*}
is bounded on $L^2 (G)$ if and only if the function $\varphi$ is $E$-essentially bounded.
If $\mathcal{L}$ is left-invariant, then so is the operator $\varphi(\mathcal{L})$.

In \cite{AkRu20}, it was shown that the $L^p$-$L^q$ norms of $\varphi(\mathcal{L})$ depend essentially only on the growth of the spectral projections $\tau(E_{(0, s)}(|\mathcal{L}|))$ of $\mathcal{L}$ for any trace $\tau$ on $\vnrp$, in particular \eqref{eq:trace}. More precisely, the result \cite[Thm.~1.2]{AkRu20} asserts:

\begin{thm}[\cite{AkRu20}] \label{thm:spectralmultiplier_intro}
Let $G$ be a locally compact separable unimodular group and let $\mathcal{L}$ be a left-invariant 
operator on $G$. Assume that $\varphi : [0, \infty) \to \mathbb{R}$ is a monotonically decreasing continuous function such that
\begin{align*}
\varphi(0) &= 1, \\ 
\lim_{s \to \infty} \varphi(s) &= 0.
\end{align*}
Then
\begin{align} \label{eq:LpLqnorm}
\| \varphi(|\mathcal{L}|) \|_{L^p(G) \to L^q(G)} \lesssim \sup_{s>0} \varphi(s) [\tau (E_{(0, s)} (|\mathcal{L}|))]^{\frac{1}{p} - \frac{1}{q}},
\end{align}
for every $1 < p \leq 2 \leq q < \infty$.
\end{thm}

Theorem \ref{thm:spectralmultiplier_intro} has the following simple consequences, e.g., see \cite[Cor.~1.3]{AkRu20}. 

\begin{cor} \label{cor:spectralmultiplier_intro}
Let $G$ be a locally compact separable unimodular group and let $\mathcal{L}$ be a positive left-invariant operator on $G$.  Suppose that there exists an $\alpha \in \mathbb{R}$ such that 
\begin{align} \label{eq:assumption_corollary}
\tau (E_{(0, s)} (\mathcal{L})) \lesssim s^{\alpha}, \quad s \to \infty.
\end{align}
Then, for any $1 < p \leq 2 \leq q < \infty$, there is a constant $C = C(p, q) > 0$ such that
\begin{align}
\| e^{-s \mathcal{L}} \|_{L^p(G) \to L^q (G)} \leq C s^{-\alpha (\frac{1}{p} - \frac{1}{q})}, \quad s \to \infty.
\end{align}
Moreover, if
\begin{align*}
\gamma \geq \alpha \bigg( \frac{1}{p} - \frac{1}{q} \bigg),
\end{align*}
then the embeddings
\begin{align*}
\| f \|_{L^q (G)} \leq C \| (1+\mathcal{L})^{\gamma} f \|_{L^p (G)}
\end{align*}
hold for every $1 < p \leq 2 \leq q < \infty$.
\end{cor}

The number $\alpha \in \mathbb{R}$ appearing in hypothesis \eqref{eq:assumption_corollary} 
was explicitly computed in \cite[\S 7]{AkRu20} for several classes of groups and operators. 
In particular, for the sub-Laplacians on compact groups, the sub-Laplacian on the Heisenberg group $\mathbb{H}_n$, and some higher order Rockland operators on $\mathbb{H}_n$. Furthermore, values of $\alpha$ for general positive Rockland operators on graded groups have been given in \cite[\S 8]{RoRu18}.
For all of these examples, the asymptotics were shown to be sharp. The proofs, however, crucially rely on results that are particular for the specific settings and operators under consideration, such as \cite{tERo97,HaKo}.

The main goal of this note is to substantially extend the range of applicability, for which \cite[Thm.~1.2]{AkRu20} and \cite[Cor.~1.3]{AkRu20} hold true, by exploiting the fact that in each previously known example the operator $\mathcal{L}$ can be interpreted as an instance of the so-called \textit{weighted subcoercive operators} defined by ter Elst and Robinson. In their seminal paper~\cite{tERo98}, the authors proved Gaussian-type heat kernel estimates for subcoercive operators of arbitrarily high order, previously only known for the heat kernels of sub-Laplacians and for Rockland operators on graded groups. The spectral theory of systems of pairwise commuting weighted subcoercive operators was later developed by Alessio Martini~\cite{Ma11}. Among others, the author proved estimates of the spectral projectors in terms of heat kernel estimates. The following proposition is a special case of~\cite[Prop.~3.11]{Ma11}, in which the system is given by a single subcoercive operator.

\begin{prop}[\cite{Ma11}] \label{prop:martini}
Let $G$ be a connected unimodular Lie group with Lie algebra $\mathfrak{g}$.
Suppose that $\mathcal{L}$ 
is an $m$-th order weighted subcoercive operator on $G$. Then 
\begin{align} \label{est:spectral_growth}
\tau(E_{(0, s)}(|\mathcal{L}|)) \lesssim s^{\frac{\hdim}{m}}, \quad s \to \infty, 
\end{align}
where $\hdim$ is the \textit{local dimension} of $G$ relative to the chosen weighted structure on $\mathfrak{g}$.

In particular, this holds true for every left-invariant hypoelliptic (positive) sub-Laplaci\-an $\mathcal{L}$ on $G$, in which case $\hdim$ equals the Hausdorff dimension of $G$ and $m = 2$, as well as for every positive Rockland operator $\mathcal{L}$ on a graded group $G$, in which case $\hdim$ equals the homogeneous dimension of $G$ and $m$ equals the homogeneous order of $\mathcal{L}$.
\end{prop}

Combining Theorem~\ref{thm:spectralmultiplier_intro} and Proposition~\ref{prop:martini}, we obtain the main result of this note:

\begin{thm} \label{thm:main}
Let $G$ be a connected unimodular Lie group with Lie algebra $\mathfrak{g}$ and let $\mathcal{L}$ 
be an $m$-th order positive weighted subcoercive operator on $G$.
Assume that $\varphi : [0, \infty) \to \mathbb{R}$ is a monotonically decreasing continuous function such that
\begin{align*}
\varphi(0) &= 1, \\ 
\lim_{s \to \infty} \varphi(s) &= 0.
\end{align*}
Then, for any $1 < p \leq 2 \leq q < \infty$, there is a constant $C = C(p, q) > 0$ such that
\begin{align} \label{eq:LpLqnorm*}
\| \varphi(|\mathcal{L}|) \|_{L^p(G) \to L^q(G)} \leq C \sup_{s>0} \varphi(s) \hspace{2pt} s^{\frac{\hdim}{m} (\frac{1}{p} - \frac{1}{q})}, \quad s \to \infty,
\end{align}
where $\hdim$ is the \textit{local dimension} of $G$ relative to the chosen weighted structure on $\mathfrak{g}$. In particular,
\begin{align} \label{eq:heat}
\| e^{-s \mathcal{L}} \|_{L^p(G) \to L^q (G)} \leq C s^{- \frac{\hdim}{m} (\frac{1}{p} - \frac{1}{q})}, \quad s \to \infty.
\end{align}
Moreover, if
\begin{align*}
\gamma \geq \frac{\hdim}{m} \bigg( \frac{1}{p} - \frac{1}{q} \bigg),
\end{align*}
then the embeddings
\begin{align} \label{eq:embeddings}
\| f \|_{L^q (G)} \leq C \| (1+\mathcal{L})^{\gamma} f \|_{L^p (G)}
\end{align}
hold for every $1 < p \leq 2 \leq q < \infty$.
\end{thm}

We stress that the asymptotics \eqref{est:spectral_growth} provide a refinement of the $L^p$-$L^q$ norms \eqref{eq:LpLqnorm} in Theorem~\ref{thm:spectralmultiplier_intro} and that it shows that the hypothesis 
\eqref{eq:assumption_corollary} is automatically satisfied for all weighted subcoercive operators on a unimodular Lie group. Moreover, the estimate in \eqref{est:spectral_growth} is sharp as will be explained below.

The proof of Proposition~\ref{prop:martini} consists of bounding the spectral traces 
by the $L^2$-norm of the heat kernel associated with $\mathcal{L}$ at time $t = 1/s$, making explicit use of the Gaussian-type estimates \cite{tERo98, Ma11}. In fact, the subcoercivity gives a characterization of the generators of semigroups with kernels satisfying such Gaussian bounds (see \cite{tERo98}) 
and provides therefore a natural and general setting in which the proof method works. Since the $L^2$-estimate of the heat kernel in the proof of~\cite[Prop.~3.11]{Ma11} is not explicitly worked out, we provide a short proof of Proposition~\ref{prop:martini} in Section~\ref{sec:proof} in view of its importance for this note.

\section{Weighted Subcoercive Operators} \label{section:subco}

This section provides preliminaries on weighted subcoercive operators 
due to ter Elst and Robinson \cite{tERo98} and Martini~\cite{Ma11}. More details can be found in \cite{tERo98, MaTh, Ma11, dungey2003analysis, elst1994weighted}. 

\subsection{Weighted Lie algebras and their contractions}
Let $G$ be a $d$-dimensional connected Lie group with Lie algebra $\mathfrak{g}$. An algebraic basis $X_1, \ldots, X_{d'}$ of $\mathfrak{g}$ is a set of linearly independent elements which together with their multi-commutators span $\mathfrak{g}$. An algebraic basis $X_1, \ldots, X_{d'}$ is said to be a \textit{weighted algebraic basis} if we associate to each vector $X_j$, $j = 1, \ldots, d'$, a so-called \textit{weight} $w_j \in [1, \infty)$. We assume that 
$\bigcap_{j = 1}^{d'} w_j \mathbb{N} \neq \emptyset$ and set
\[
w = \min_{s \in [1,\infty)} \big\{s \in w_j \mathbb{N}, \; \forall j \in \{1, ..., d'\} \big\}.
\]
We denote by $J(d')$ the set of finite sequences of elements of $\{1, ..., d'\}$. 
For arbitrary $\alpha = (\alpha_1, ..., \alpha_n) \in J(d')$, let $|\alpha| = n$ denote the Euclidean length of $\alpha$ 
and denote by  $\| \alpha \| : = \sum_{j = 1}^{d'} w_{\alpha_j}$ the weighted length of $\alpha$. 
Moreover, we define
\[
X^{\alpha} = X_{\alpha_1} \cdots X_{\alpha_n}
\]
as an element of the universal enveloping algebra $U (\mathfrak{g})$ of $\mathfrak{g}$. 

A given weighted algebraic basis defines a filtration $\{ F_\lambda \}_{\lambda \in \mathbb{R}}$ on $\mathfrak{g}$,
\begin{align*}
F_\lambda := \mathrm{span} \bigl \{ [ \, \ldots \, [X_{\alpha_1}, X_{\alpha_2}], \ldots, X_{\alpha_n}] \mid \text{ non-empty } \alpha \in J(d'), \| \alpha \| \leq \lambda \bigr \},
\end{align*}
which satisfies
\begin{align*}
[F_\lambda, F_\mu] \subseteq F_{\lambda + \mu}, \hspace{10pt} F_\lambda = \bigcap_{\mu > \lambda} F_\mu,  \hspace{10pt} \bigcup_{\lambda \in \mathbb{R}} F_\lambda = \mathfrak{g}.
\end{align*}
A weighted algebraic basis of $\mathfrak{g}$ is said to be \textit{reduced} if for all $\lambda \in \mathbb{R}$
\begin{align*}
\mathrm{span} \bigl \{ X_{\alpha_j} \mid w_j = \lambda \bigr \} \bigcap F^-_\lambda = \{ 0 \}, \text{ where } F^-_\lambda := \bigcup_{\mu < \lambda} F_\mu = \{ 0 \}.
\end{align*}
Given a weighted basis and a filtration, there always exists a reduced weighted basis $X_1, \ldots, X_{d''}$ of $\mathfrak{g}$ which defines the same filtration. A Lie algebra equipped with a reduced weighted basis is said to be \textit{weighted}. Any such Lie algebra admits an associated homogeneous Lie algebra $\mathfrak{g}_*$ together with a decomposition
\begin{align*}
\mathfrak{g}_* = W_{\lambda_1} \oplus \ldots \oplus W_{\lambda_k}
\end{align*}
into a direct sum of eigenspaces of the homogeneous dilations, defined by $W_{\lambda_j} := F_{\lambda_j} \neq F^-_{\lambda_j}$, for weights $1 \leq \lambda_1 < \ldots < \lambda_k$ determined by $F_{\lambda_j} / F^-_{\lambda_j}$, $j = 1, \ldots, k$. The weights $w_1, \ldots, w_{d''}$ are among the weights $\lambda_1 < \ldots < \lambda_k$ since $X_1, \ldots, X_{d''}$ is reduced. If $\bar{X}_j$ denotes the projection onto $W_{w_j}$ of $X_j \in F_{w_j}$, then $\bar{X}_1, \ldots, \bar{X}_{d''}$ is an adapted basis of $\mathfrak{g}_*$ with the same weights $w_1, \ldots, w_{d''}$. The homogeneous Lie algebra $\mathfrak{g}_*$ equipped with the adapted basis $\bar{X}_1, \ldots, \bar{X}_{d''}$ is called the contraction of the weighted Lie algebra $\mathfrak{g}_*$.

A \textit{weighted Lie group} is a connected Lie group $G$ whose Lie algebra $\mathfrak{g}$ is weighted. The \textit{contraction} $G_*$ of a weighted Lie group is the homogeneous Lie group whose Lie algebra is $\mathfrak{g}_*$.

Any weighted Lie group can be equipped with a control distance $d_*: G \times G \to [0, \infty)$ with the following properties: the left Haar measure of the metric balls $B_r(e)$ of radius $r$ around the unit element $e \in G$ is given by $|B_r(e)| \sim r^{\hdim}$ for $r \leq 1$, where $\hdim$ is the homogeneous dimension of the contraction $G_*$, and in the case that $G$ is homogeneous the control modulus defined by $|x|_* := d_*(e, x)$ coincides with a homogeneous quasi-norm for small $x$; for $r \geq 1$, the metric coincides with the Carnot-Carath\'{e}odory metric associated with the 
reduced (unweighted) basis vector fields $X_1, \ldots, X_{d''}$ on $G$, which necessrily satisfy H\"{o}rmander's condition. This control modulus is a connected distance. Note that for large radii $r$ the volume growth is controlled by $| B_r(x)| \lesssim e^{\beta r}$ for some $\beta > 0$. (For details we refer to \cite{MaTh, Ma11}.)

\subsection{Differential operators} Let $G$ be a weighted Lie group and fix a reduced weighted algebraic basis $X_1, ..., X_{d''}$ of $\mathfrak{g}$.
A function $C : J(d'') \to \mathbb{C}$ such that $C(\alpha) = 0$ if $\| \alpha \| > m$ but $C(\alpha) \neq 0$ for at least one $\alpha \in J(d'')$ with $\| \alpha \| = m$ is said to be an $m$-th order \textit{form}. The \textit{principal part} of $C$ is the form $P: J(d'') \to \mathbb{C}$ given by the sum of terms of $C$ of order $m$:
\begin{align*}
P(\alpha) =
\left \{ \begin{array}{lcr}
& C(\alpha) & \text{ if } \| \alpha \| = m, \\
& 0 & \text{ otherwise.}
\end{array} \right.
\end{align*}
A form $C$ is said to be \textit{homogeneous} if $C = P$. The \textit{adjoint} $C^+$ of a form $C$ is defined by $C^+(\alpha) := (-1)^{| \alpha |} \overline{C(\alpha_*)}$, where $\alpha_* := (\alpha_k, \ldots, \alpha_1)$ if $\alpha = (\alpha_1, \ldots, \alpha_k) \in J(d'')$. The form $C$ is said to be \textit{symmetric} if $C^+ = C$.

For a given form $C$, we can consider the $m$-th order operators 
\begin{align*}
d\rho_G (C) = \sum_{\alpha \in J (d'')} C(\alpha) d\rho_G (X^{\alpha}) 
\end{align*}
with domain $D( d\rho_G (C)) = \bigcap_{\| \alpha \| \leq m} D(d\rho_G(X^{\alpha}))$. The associated semi-norms are defined by
\[
N_{s}(f) = \max_{\substack{\alpha \in J(d''),\\ \| \alpha \| \leq s } } \big\| d \rho_G (X^{\alpha}) f \big\|_{L^2 (G)}
\]
for $s \in \mathbb{R}$ with $s \geq 0$. If we consider the right regular representation $\rho_G$ on $L^p(G)$, $1 \leq p < \infty$, instead of $L^2(G)$, we denote the analogue of $D( d\rho_G (C))$ by $L^{p; m}(G)$ and set $\bigcap_{m = 1}^\infty L^{p; m}(G) =: L^{p; \infty}(G)$.

An $m$-th order form $C$ is said to be a \textit{weighted subcoercive form}, 
with respect to $X_1, ..., X_{d''}$, if $m \in 2 w \mathbb{N}$  and if $d\rho_G(C)$ satisfies a local \textit{G\aa rding inequality}: 
there exist $\mu > 0$, $\nu \in \mathbb{R}$ and an open neighborhood $V$ of the identity $e_G \in G$ such that
\begin{align*}
\operatorname{Re}( \langle f, d\rho_G(C) f \rangle) \geq \mu \big( N_{m/2} (f) \big)^2- \nu \| f \|_{L^2(G)}^2
\end{align*}
for all $f \in \mathcal{D}(G)$ with support $\supp f \subseteq V$. 

We say that an operator $\mathcal{L} = d\rho_G (C)$ on $G$ is \textit{weighted subcoercive} if 
$C$ is a weighted subcoercive form.

We recall that an operator $\mathcal{L} = d\rho_G (C) = d\rho_G (P)$ on a homogeneous group $G$ is called Rockland if the (principal) form $C = P$ is homogeneous with respect to the given homogeneous dilations on $G$ and $d \pi (C)$ is injective on the space of smooth vectors $\mathcal{H}^\infty$ for every non-trivial unitary irreducible representation $\pi \in \widehat{G}$. The existence of a Rockland operator on a homogeneous group $G$ implies that the dilation weights $w_1, \ldots, w_d$ can be jointly rescaled to be positive integers and that $\mathfrak{g}$ possesses a gradation.

The following theorem collects several key properties of weighted subcoercive operators. 

\begin{thm}[\cite{tERo98, Ma11}] \label{thm:subcoercive}
Let $G$ be a weighted Lie group and $C$ an $m$-th order form with principal part $P$. 
Then the following are equivalent:
\begin{enumerate}[(i)]
	\item $\mathcal{L} = d\rho_G(C)$ is a weighted subcoercive operator on $G$; 
	\item $d\rho_{G_*}(P + P^+)$ is a positive Rockland operator on the contraction $G_*$.
\end{enumerate}
\medskip
\noindent If these conditions are satisfied, then the following properties hold:
\begin{enumerate}[(a)]
\item The continuous semigroup $S = \{S_t \}_{t \geq 0}$ generated by $\mathcal{L}$ 
has a smooth kernel $k_t \in L^{1;\infty}(G) \cap C_0^{\infty} (G)$  such that
\[
d\rho_G (X^{\alpha}) S_t f =  \int_G (X^{\alpha} k_t) (x) \rho_G(x^{-1}) f \; dx
\]
for all $\alpha \in J(d')$ and $f \in L^2 (G)$. 
\item For all $\alpha \in J(d')$, there exist $b, c > 0$ and $\omega \geq 0$ such that
\begin{align}
\left | k_t(x) \right | \leq c t ^{-\frac{\hdim}{m}} \hspace{2pt} \exp (\omega t) \hspace{2pt} \exp \Bigl ( -b \bigl ( \frac{{|x|}_*^m}{t} \bigr )^{\frac{1}{m-1}} \Bigr ), \quad x \in G, t > 0, \label{eq:est_hk_tERo}
\end{align}
where $\hdim$ is the homogeneous dimension of the contraction $G_*$ 
and $|\cdot|_{*}$ is the control modulus associated with the control distance $d_*$.
\item The function $k : \mathbb{R} \times G \to \mathbb{C}$ defined by 
$(t, x) \mapsto k_t (x)$ for $t > 0$ and $k_t = 0$ for $t\leq0$ satisfies the heat equation 
\[
((\partial_t + d\rho_G (C)) k_t )(x) = \delta(t) \delta(x)
\]
as distributions, for all $(t, x) \in \mathbb{R} \times G$.
\end{enumerate}

\end{thm}

\section{Proof of Proposition~\ref{prop:martini}} \label{sec:proof}

In view of the importance of Proposition~\ref{prop:martini} (a special case of~\cite[Prop.~3.11]{Ma11}) for this note, we will give a short proof. The argument can essentially be split into two statements of independent interest, which concern: 1) bounding the spectral traces by the $L^2$-norm of the heat kernel when $t \to 0$; 2) an explicit computation to bound the latter.

In order to relate spectral traces to heat kernels, one can employ the abstract Plancher\-el theorem \cite[Thm.~1.6.1]{gangolli1988harmonic}; see also \cite[Thm~3.10]{Ma11} and \cite[Prop.~3]{christ1991lp}.

\begin{prop} \label{prop:trace_heat}
Let $G$ be a connected unimodular Lie group and let $\mathcal{L}$ 
be a weighted subcoercive operator on $G$. 
For $t>0$, let $k_t$ denote the heat kernel associated with $\mathcal{L}$. 
Then 
\begin{align*}
\tau (E_{(0, s)} (|\mathcal{L}|)) \lesssim \| k_{1/s} \|^2_{L^2 (G)}
\end{align*}
for all $s > 0$. 
\end{prop}
\begin{proof}
For $s > 0$, set $E_s := E_{(0, s)} (|\mathcal{L}|)$. Let $K_{E_s}$ denote the convolution kernel of $E_s$. 
By the above-mentioned abstract Plancherel theorem, the map $s \mapsto \tau(E_s): \mathbb{R} \to [0, \infty)$ defines the distribution function of a regular Borel measure $\mu$ on $\mathbb{R}$ which for $\varphi \in L^\infty(\mathbb{R}, \mu)$ satisfies
\begin{align*}
\| K_{\varphi(\mathcal{L})} \|_{L^2(G)}^2 = \| \varphi \|_{L^2 (\mathbb{R}, \mu)}^2.
\end{align*}
For $\varphi = \chi_{(0, s)}$ this yields the estimate
\begin{align}
\begin{split} \label{est:Plancherel}
\tau(E_s) &= \tau(E_s^2) = \| K_{E_s} \|_{L^2 (G)}^2 = \| \chi_{(0, s)} \|^2_{L^2 (\mathbb{R}, \mu)} \\
&\lesssim \| e^{-(1/s) \, \cdot \,} \|_{L^2 (\mathbb{R}, \mu)}^2 = \| k_{1/s} \|_{L^2 (G)}^2,
\end{split}
\end{align}
as desired.
\end{proof}

\begin{rem}
A different proof of Proposition \ref{prop:trace_heat} 
can be obtained via \cite[Prop.~2.1]{cowling2019hausdorff}, 
which relies on the theory of non-commutative $L^p$-spaces. 
This was pointed out to us by 
Alessio Martini.
\end{rem}

The second statement is a direct consequence of the Gaussian-type heat kernel estimates due to Theorem~\ref{thm:subcoercive}~(iii)~(\cite{tERo98, Ma11}). Note that it can also be deduced by interpolating the $L^1$ and $L^\infty$ estimates for the heat kernel given by \cite[Thm.~2.3]{Ma11}.

\begin{prop} \label{lem:L2_bound_hk}
Let $G$ and $\mathcal{L}$ 
be as in Proposition~\ref{prop:martini}.
For $t>0$, let $k_t$ denote the heat kernel associated with $\mathcal{L}$. Then
\begin{align*}
\| k_t \|^2_{L^2(G)} \lesssim t^{-\frac{\hdim}{m}}, \quad  t \to 0.
\end{align*}
\end{prop}

\begin{proof}
Since for small $t > 0$ and Theorem~\ref{thm:subcoercive}~(b),
\begin{align*}
	\int_G |k_t(x)|^2 \, dx \lesssim t^{-\frac{2\hdim}{m}} \underbrace{\int_G \exp \Bigl (-2b {\Bigl (\frac{ |x|_*^m}{t} \Bigr)}^{\frac{1}{m-1}} \Bigr ) \, dx}_{=:I(t)},
\end{align*}
it suffices to show that $I(t) \lesssim t^{\frac{\hdim}{m}}$ as $t \to 0$. To do so, we cut $G$ into a dyadic ball at the origin $x = e$ and dyadic annuli outside the ball, all of whose radii are functions of $t > 0$. Thus, consider the sets
\begin{align*}
	A_{0, t} &:= \{ x \in G \mid |x|_*^m \leq 2t \} = B_{\sqrt[m]{2 t}}(e) \\
	A_{j, t} &:= \{ x \in G \mid 2^j t < |x|_*^m \leq 2^{j+1} t \}, \quad j \in \mathbb{N}.
\end{align*}
This yields a $t$-dependent partition $G = \bigsqcup_{j \in \mathbb{N}_0} A_{j, t}$ and the corresponding splitting
\begin{align*}
	I(t) =  \sum_{j = 0}^\infty I_j(t) := \sum_{j = 0}^\infty \int_{A_{j, t}} \exp \Bigl (-2b {\Bigl (\frac{ |x|_*^m}{t} \Bigr)}^{\frac{1}{m-1}} \Bigr ) \, dx.
\end{align*}
For each $j \in \mathbb{N}$
\begin{align*}
	2^j < \frac{|x|_*^m}{t} \leq 2^{j+1} \quad \Rightarrow \quad \sup_{x \in A_{j, t}} \exp \Bigl (-2b{\Bigl (\frac{ |x|_*^m}{t} \Bigr)}^{\frac{1}{m-1}} \Bigr ) = \exp \bigl (-2b \, 2^{\frac{j}{m-1}} \bigr )
\end{align*}
implies that
\begin{align*}
	\int_{A_{j, t}} \exp \Bigl (-2b {\Bigl (\frac{ |x|_*^m}{t} \Bigr)}^{\frac{1}{m-1}} \Bigr ) \, dx  &\leq \exp \bigl (-2b \, 2^{\frac{j}{m-1}} \bigr ) \, | A_{j, t} | \\
	&\leq \exp \bigl (-2b \, 2^{\frac{j}{m-1}} \bigr ) \, |B_{\sqrt[m]{2^{j+1} t}}(e)|,
\end{align*}
while
\begin{align*}
	I_0(t) := \int_{A_{0, t}} \exp \Bigl (-2b {\Bigl (\frac{ |x|_*^m}{t} \Bigr)}^{\frac{1}{m-1}} \Bigr ) \, dx \leq | A_{0, t} |.
\end{align*}
Furthermore, we recall that the measure of balls of large radii is controlled by $| B_r(x)| \lesssim e^{\beta r}$ for some $\beta > 0$. Since for any fixed $t >0$ the exponent $2b{ (\frac{ r^m}{t} )}^{\frac{1}{m-1}}$ is larger than $\beta r$ for large $r > 0$, it follows that $I(t)$ is finite for all $t > 0$.
We now use the fact that as $t \to 0$ most of the mass of $I(t)$ is concentrated inside an arbitrary small ball around $x = e$ and that for small balls the volume growth is determined by the homogeneous dimension $\hdim$. Thus, for every $\varepsilon > 0$ there exists some $t' > 0$ such that for all $t \in (0, t')$
\begin{align*}
	I(t) \leq (1 + \varepsilon) \int_{B_{\frac{1}{2}}(e)} \exp \Bigl (-2b {\Bigl (\frac{ |x|_*^m}{t} \Bigr)}^{\frac{1}{m-1}} \Bigr ) \, dx.
\end{align*}
So, it suffices to integrate over the union of sets $A_{j ,t}$, for which $A_{0, t}, A_{1, t}, \ldots,$ $A_{N(t)-1, t} \subset B_{\frac{1}{2}}(e)$ and $B_{\frac{1}{2}}(e) \subseteq A_{N(t), t} \subseteq B_{1}(e)$ since the integral over the infinitely many remaining sets $A_{N(t)+1, t}, A_{N(t)+2, t}, \ldots$ is absorbed by the factor $(1 + \varepsilon)$. Moreover,  we have $| A_{j, t} | \leq |B_{\sqrt[m]{2^{j+1} t}}(e)| \asymp (2t)^{\frac{\hdim}{m}} \, 2^{\frac{(j+1) \hdim}{m}}$ since for small radii $|\{ x \in G \mid |x|_* \leq r \} | \asymp r^{\hdim}$. Hence, for all $t \in (0, t')$ we have
\begin{align*}
	I(t) &\leq (1 + \varepsilon) \sum_{j = 0}^{N(t)} I_j(t) \\ &\lesssim (1 + \varepsilon) (2t)^{\frac{\hdim}{m}} \Bigl ( 1 + \sum_{j = 1}^{N(t)} \exp \bigl (-2b \, 2^{\frac{j}{m-1}} \bigr ) 2^{\frac{(j+1) \hdim}{m}} \Bigr ) \\
	&\lesssim (1 + \varepsilon) (2t)^{\frac{\hdim}{m}} \Bigl ( 1 + \sum_{j = 1}^{\infty} \exp \bigl (-2b \, 2^{\frac{j}{m-1}} \bigr ) 2^{\frac{(j+1) \hdim}{m}} \Bigr ) \\ &\lesssim t^{\frac{\hdim}{m}},
\end{align*}
which completes the proof.
\end{proof}

\begin{proof}[Proof of Proposition~\ref{prop:martini}]
The spectral growth for large $s$ is asymptotically bounded by the squared $L^2$-norm of the heat kernel for small times $t = 1/s$, as shown in Proposition~\ref{prop:trace_heat}, and the quantitative estimate for the $t$-dependent $L^2$-norm is given by Proposition~\ref{lem:L2_bound_hk}. The arguments which justify the inclusion of the examples will be given in Section~\ref{section:Ex}.
\end{proof}

\begin{proof}[Proof of Therem~\ref{thm:main}]
Combining Theorem~\ref{thm:spectralmultiplier_intro} and Proposition~\ref{prop:martini}, we obtain \eqref{eq:LpLqnorm*}. Then Corollary~\ref{cor:spectralmultiplier_intro} implies \eqref{eq:heat} and \eqref{eq:embeddings}.
\end{proof}

\section{Examples} \label{section:Ex}

We conclude this note with a handful of exemplary classes of subcoercive operators.

\begin{ex} \label{ex:Rockland}
The prototypical example of a weighted subcoercive operator is a \textit{positive essentially self-adjoint Rockland operator} on a graded group $G$.

To begin with, consider a $d$-dimensional graded group $G$ equipped with a family of homogeneous dilations $\{ \gamma_t \}_{t > 0}$ with eigenbasis $Y_1, \ldots, Y_d$ and corresponding weights $v_1, \ldots, v_d \in \mathbb{N}$.\footnote{If $G$ is graded, one can always jointly rescale the dilations' weights so that they are integers satisfying $1 \leq v_1 \leq \ldots \leq v_n$.} Furthermore, consider $G$ as a weighted group with respect to the weights $v_1, \ldots, v_d$. If for a given reduced algebraic basis $X_1, \ldots, X_{d''}$ of $\mathfrak{g}$ and a form $C : J(d'') \to \mathbb{C}$ the operator $\mathcal{L} = d\rho_G(C)$ is a positive essentially self-adjoint Rockland operator of homogeneous order $m \in 2 w \mathbb{N}$, with $w := \mathrm{lcm}(v_1, \ldots, v_{d''})$, then by \cite[Thm.~2.5]{tERo97} the operator $\mathcal{L}$ satisfies a G\aa rding inequality. By Theorem~\ref{thm:subcoercive}~(ii), it is a weighted subcoercive operator on $G$. Note that in this case the homogeneous contraction $G_*$ coincides with the original group $G$ and $Q_*$ is the homogeneous dimension of $G$.

More generally, one may consider an arbitrary positive essentially self-adjoint Rockland operator $\mathcal{L}$ on the given graded group $G$. Then the operator can be written as $\mathcal{L} = d\rho_G(C)$ for some homogeneous form $C : J(d) \to \mathbb{C}$ of degree $m \in \mathbb{N}$ in the eigenbasis $Y_1, \ldots, Y_d$, but it does not follow automatically that the form $C$ expressed with respect to the reduced algebraic basis $X_1, \ldots, X_{d''}$ is weighted subcoercive. However, by \cite[Lem.~2.2]{tERo97} and \cite[Lem.~2.4]{tERo97}, there exists an eigenbasis of the dilations, $Z_1, \ldots, Z_d$, together with weights $u_1, \ldots, u_d$ and a $d' \in \{1, \ldots, d \}$, such that $Z_1, \ldots, Z_{d'}$ is a reduced algebraic basis of $\mathfrak{g}$ with $\mathrm{span} \{ Z_1, \ldots, Z_{d'} \} \cap [\mathfrak{g}, \mathfrak{g}] = 0$; with respect to this basis every positive essentially self-adjoint Rockland operator $\mathcal{L}$ on $G$ possesses a homogeneous form $P$ of degree $m$ such that $\mathcal{L} = d\rho_G(P)$ and $m \in 2 u_j \mathbb{Z}$, $j = 1, \ldots, d'$. By \cite[Thm.~2.5]{tERo97}, the operator $\mathcal{L}$ satisfies a G\aa rding inequality and is thus weighted subcoercive on the weighted group $G$. Again, the homogeneous contraction $G_*$ coincides with the original group $G$ and $Q_*$ is the homogeneous dimension of $G$.

So, weighted subcoercivity is in particular satisfied for all Rockland operators of the form
\begin{align}
\mathcal{L} = \sum_{j=1}^{d'} (-1)^{\frac{m}{2 u_j}} c_j \hspace{1pt} d\rho_G(Z_j)^{\frac{m}{u_j}}, \label{eq:RO1}
\end{align}
with coefficients $c_1, \ldots, c_{d'} > 0$, or
\begin{align}
\mathcal{L} = \sum_{j=1}^{d} (-1)^{\frac{m}{2 w_j}} c'_j \hspace{1pt} d\rho_G(X_j)^{\frac{m}{w_j}}, \nonumber 
\end{align}
with coefficients $c'_1, \ldots, c'_{d} > 0$.

In the case of an arbitrary Rockland operator on $G$, one cannot automatically employ the heat kernel estimates for weighted subcoercive operators from Theorem~\ref{thm:subcoercive}~(b) in order to deduce the bound \eqref{est:spectral_growth}, for the reasons laid out above. However, one can always combine the heat kernel estimates for Rockland operators given by \cite[Thm.~1.1]{AutERo} with \cite[Prop.~3.11]{Ma11} to deduce the desired estimate
\begin{align*}
	\tau (E_{(0, s)} (\mathcal{L})) \lesssim s^{\frac{\hdim}{m}}, \quad s \to \infty.
\end{align*}
Note that one can even show that
\begin{align}
\tau (E_{(0, s)} (\mathcal{L})) = C(\mathcal{L}) \, s^{\frac{\hdim}{m}}, \quad s > 0, \label{eq:growth_graded_alternative}
\end{align}
with $C(\mathcal{L}) = \| K_{E_{(0, s)}(\mathcal{L})} \|_{L^2(G)}$, from relating the interaction of the group's dilations $\{ \gamma_t \}_{t > 0}$ to the spectral calculus of $\mathcal{L}$:
\begin{align*}
\bigl ( E_{(0, s)}(\mathcal{L}) \bigr )^{\frac{1}{2}}f = s^{\frac{\hdim}{m}} \, f * ( K_{E_{(0, 1)}(\mathcal{L})} \circ \gamma_{s^{\frac{1}{m}}});
\end{align*}
computing the spectral trace of $E_{(0, s)}(\mathcal{L})$, one then immediately obtains \eqref{eq:growth_graded_alternative}. A more general version of such relations on connected Lie groups, \cite[Prop.~3.18]{Ma11}, also immediately implies \eqref{eq:growth_graded_alternative} for graded groups due the fact that $\gamma_{s^{\frac{1}{m}}}$ is a Lie group automorphism.
Alternatively, as in~\cite{RoRu18}, one can deduce
\begin{align}
\tau (E_{(0, s)} (\mathcal{L})) \asymp s^{\frac{\hdim}{m}}, \label{eq:growth_graded}
\end{align}
by using Kirillov's orbit method and the spectral estimates for $\pi(\mathcal{L})$, $\pi \in \widehat{G}$, established in~\cite{tERo97}.

\end{ex}

\begin{ex} \label{ex:subL_stratified}
A special case of \eqref{eq:RO1} is the \textit{negative homogeneous sub-Laplacian}
\begin{align}
-\mathcal{L}_G = -\sum_{j=1}^{d'} d\rho_G(X_j)^2 \label{eq:ROsubL}
\end{align}
on a stratified group $G$ equipped with the canonical homogeneous dilations (also known as a Carnot group), where the basis $X_1, \ldots, X_{d'}$ of the first stratum of $\mathfrak{g}$ is automatically a reduced weighted algebraic basis of $\mathfrak{g}$ with all weights equal $1$. Particularly well known is the sub-Laplacian on the Heisenberg group $\mathbb{H}_1$
\begin{align}
\mathcal{L}_{\mathbb{H}_1} &= \bigl( \partial_x - \frac{1}{2} y \hspace{2pt} \partial_t \bigr)^2 + \bigl( \partial_y + \frac{1}{2} x \hspace{2pt} \partial_t \bigr)^2, \label{eq:subLHeisenberg}
\end{align}
whose heat kernel has been explicitly known for a long time due to~\cite{Hu76}. Since $m = 2$ and $\hdim = 4$, the heat kernel estimate recovers the growth rate
\begin{align*}
\tau (E_{(0, s)} (-\mathcal{L}_{\mathbb{H}_1})) \asymp s^{2},
\end{align*}
proved in \cite{AkRu20}. As a special case of \eqref{eq:growth_graded_alternative}, it is also a direct consequence of  \cite[Prop.~3.18]{Ma11}.
\end{ex}

\begin{ex} \label{ex:subL_unimodular}
Let $G$ be a $d$-dimensional connected unimodular Lie group and let $X_1, \ldots, X_{d'}$ be an algebraic basis of its Lie algebra $\mathfrak{g}$, that is, the associated left-invariant vector fields $d\rho_G(X_1), \ldots, d\rho_G(X_{d'})$ satisfy H\"{o}rmander's condition. If $C$ denotes the negative sum of squares of $X_1, \ldots, X_{d'}$, then the operator $$-\mathcal{L}_G := d\rho_G(C)$$ forms a left-invariant hypoelliptic sub-Laplacian on $G$. Since the system $X_1, \ldots, X_{d'}$ forms an algebraic basis of $\mathfrak{g}$, one obtains a contraction of $\mathfrak{g}$ to a stratified Lie algebra $\mathfrak{g}_*$ by assigning the weights $w_1 = \ldots = w_{d'} = 1$. By \cite[Thm.~10.1]{tERo98}, the operator $d\rho_G(C)$ is weighted subcoercive if and only if $d\rho_{G_*}(C)$ is weighted subcoercive, which is the case for an operator of the form \eqref{eq:ROsubL}. The Hausdorff dimension of $G$ with respect to the Carnot-Carath\'{e}odory distance defined by $X_1, \ldots, X_{d'}$, which equals  the homogeneous dimension $\hdim$ of $G_*$, determines the spectral growth bound
\begin{align*}
	\tau (E_{(0, s)} (-\mathcal{L}_G)) \lesssim s^{\frac{\hdim}{2}}, \quad s \to \infty.
\end{align*}

This estimate recovers the asymptotic growth rate of the eigenvalue counting function of invariant hypoellptic sub-Laplacians on connected compact Lie groups due to~\cite{HaKo}:
\begin{align}
N(\lambda) \asymp \lambda^{\frac{\hdim}{2}}, \hspace{10pt} \lambda \to \infty. \nonumber
\end{align}
\end{ex}

\begin{ex} \label{ex:subL_Selection}
Example~\ref{ex:subL_unimodular} covers a wide range of different Lie groups like, e.g., the compact groups $\mathrm{SU}(2)$ and $\mathrm{SO}(3)$, the nilpotent Heisenberg group $\mathbb{H}_1$, the simple Lie group $\mathrm{SL}(2, \mathbb{R})$, and the Euclidean motion group $\mathrm{SE}(2)$, a solvable Lie group. Although they are all unimodular, their respective Haar measures show very different growth behaviors: $\mathrm{SU}(2)$ and $\mathrm{SO}(3)$ are of bounded measure, the Haar measures of $\mathbb{H}_1$ and $\mathrm{SE}(2)$ grow polynomially, while  $\mathrm{SL}(2, \mathbb{R})$ is of exponential growth. Nevertheless all of these groups contract to the same stratified group $G_*$, the Heisenberg group $\mathbb{H}_1$\footnote{This is easily checked by their respective Lie bracket relations.}, and the local model of any sub-Laplacians on these groups is \eqref{eq:subLHeisenberg}. Hence, for all of these groups we have
\begin{align*}
\tau (E_{(0, s)} (-\mathcal{L}_G)) \lesssim s^{2}, \quad s \to \infty.
\end{align*}

Let us mention that the heat kernels for these groups were computed explicitly in \cite{AgBoGaRo}. In their analysis the authors make heavy use of the respective group Fourier transforms and, to phrase it in the language of pseudo-differential operators on Lie groups (cf.~\cite{RuTu, FiRu, Ng, CaRu}), explicitly describe the (operator-valued) symbol $\pi(\mathcal{L})$ of the sub-Laplacian $\mathcal{L}$ in question in every irreducible unitary group representation $\pi \in \widehat{G}$; in particular they describe the spectra and eigenfunctions of the operators $\pi(\mathcal{L})$.
\end{ex}

\section*{Acknowledgments}

The authors would like to thank Jordy van Velthoven for useful discussions and Alessio Martini for pointing out the link between spectral traces and heat kernels estimates in the first place. Moreover, they would like to thank the careful reviewer for their useful comments and references.

\medskip

The authors are supported by the FWO Odysseus 1 grant G.0H94.18N: Analysis and Partial Differential Equations and the FWO Senior Research Grant G022821N: Niet-commutatieve wavelet analyse. Michael Ruzhansky is also supported by the Methusalem programme of the Ghent University Special Research Fund (BOF) (Grant number 01M01021) and by the EPSRC grant EP/R003025/2.

\bibliography{Lp_Lq_Spectral_Multipliers.bib}

\begin{thebibliography}{10}

\bibitem{AgBoGaRo}
Andrei Agrachev, Ugo Boscain, Jean-Paul Gauthier, and Francesco Rossi.
\newblock The intrinsic hypoelliptic {L}aplacian and its heat kernel on
  unimodular {L}ie groups.
\newblock {\em J. Funct. Anal.}, 256(8):2621--2655, 2009.

\bibitem{AkRu20}
Rauan Akylzhanov and Michael Ruzhansky.
\newblock {$L^p$}-{$L^q$} multipliers on locally compact groups.
\newblock {\em J. Funct. Anal.}, 278(3):108324, 49, 2020.

\bibitem{AutERo}
Pascal Auscher, A.~F.~M. ter Elst, and Derek~W. Robinson.
\newblock On positive {R}ockland operators.
\newblock {\em Colloq. Math.}, 67(2):197--216, 1994.

\bibitem{CaRu}
Duv\'{a}n Cardona and Michael Ruzhansky.
\newblock Subelliptic pseudo-differential operators and {F}ourier integral
  operators on compact {L}ie groups.
\newblock {\em Preprint}, 2020.
\newblock \url{https://arxiv.org/abs/2008.09651}.

\bibitem{christ1991lp}
Michael Christ.
\newblock {$L^p$} bounds for spectral multipliers on nilpotent groups.
\newblock {\em Trans. Amer. Math. Soc.}, 328(1):73--81, 1991.

\bibitem{cowling2019hausdorff}
Michael~G. Cowling, Alessio Martini, Detlef M\"{u}ller, and Javier Parcet.
\newblock The {H}ausdorff-{Y}oung inequality on {L}ie groups.
\newblock {\em Math. Ann.}, 375(1-2):93--131, 2019.

\bibitem{dungey2003analysis}
Nick Dungey, A.~F.~M. ter Elst, and Derek~W. Robinson.
\newblock {\em Analysis on {L}ie groups with polynomial growth}, volume 214 of
  {\em Progress in Mathematics}.
\newblock Birkh\"{a}user Boston, Inc., Boston, MA, 2003.

\bibitem{FiRu}
Veronique Fischer and Michael Ruzhansky.
\newblock {\em Quantization on nilpotent {L}ie groups}, volume 314 of {\em
  Progress in Mathematics}.
\newblock Birkh\"{a}user/Springer, [Cham], 2016.

\bibitem{gangolli1988harmonic}
Ramesh Gangolli and V.~S. Varadarajan.
\newblock {\em Harmonic analysis of spherical functions on real reductive
  groups}, volume 101 of {\em Ergebnisse der Mathematik und ihrer Grenzgebiete
  [Results in Mathematics and Related Areas]}.
\newblock Springer-Verlag, Berlin, 1988.

\bibitem{HaKo}
Asma Hassannezhad and Gerasim Kokarev.
\newblock Sub-{L}aplacian eigenvalue bounds on sub-{R}iemannian manifolds.
\newblock {\em Ann. Sc. Norm. Super. Pisa Cl. Sci. (5)}, 16(4):1049--1092,
  2016.

\bibitem{Hu76}
Andrzej Hulanicki.
\newblock The distribution of energy in the {B}rownian motion in the {G}aussian
  field and analytic-hypoellipticity of certain subelliptic operators on the
  {H}eisenberg group.
\newblock {\em Studia Math.}, 56(2):165--173, 1976.

\bibitem{MaTh}
Alessio Martini.
\newblock {\em Algebras of differential operators on Lie groups and spectral
  multipliers}.
\newblock PhD thesis, Scuola Normale Superiore Pisa, 2010.

\bibitem{Ma11}
Alessio Martini.
\newblock Spectral theory for commutative algebras of differential operators on
  {L}ie groups.
\newblock {\em J. Funct. Anal.}, 260(9):2767--2814, 2011.

\bibitem{Ng}
Binh-{K}hoi Nguyen.
\newblock {\em Pseudo-{D}ifferential {C}alculus on {G}eneralized {M}otion
  {G}roups}.
\newblock PhD thesis, Imperial College London, September 2016.
\newblock \url{https://spiral.imperial.ac.uk/handle/10044/1/44081}.

\bibitem{RoRu18}
David Rottensteiner and Michael Ruzhansky.
\newblock Harmonic and anharmonic oscillators on the {H}eisenberg group.
\newblock {\em J. Math. Phys.}, 63(11):Paper No. 111509, 23, 2022.

\bibitem{RuTu}
Michael Ruzhansky and Ville Turunen.
\newblock {\em Pseudo-differential operators and symmetries}, volume~2 of {\em
  Pseudo-Differential Operators. Theory and Applications}.
\newblock Birkh\"{a}user Verlag, Basel, 2010.
\newblock Background analysis and advanced topics.

\bibitem{elst1994weighted}
A.~F.~M. ter Elst and Derek~W. Robinson.
\newblock Weighted strongly elliptic operators on {L}ie groups.
\newblock {\em J. Funct. Anal.}, 125(2):548--603, 1994.

\bibitem{tERo97}
A.~F.~M. ter Elst and Derek~W. Robinson.
\newblock Spectral estimates for positive {R}ockland operators.
\newblock In {\em Algebraic groups and {L}ie groups}, volume~9 of {\em Austral.
  Math. Soc. Lect. Ser.}, pages 195--213. Cambridge Univ. Press, Cambridge,
  1997.

\bibitem{tERo98}
A.~F.~M. ter Elst and Derek~W. Robinson.
\newblock Weighted subcoercive operators on {L}ie groups.
\newblock {\em J. Funct. Anal.}, 157(1):88--163, 1998.

\end{thebibliography}
\bibliographystyle{plain}

\end{document}